\documentclass[a4paper,12pt,reqno]{amsart}
\usepackage{amsfonts, color}
\usepackage{amsmath}
\usepackage{amssymb}
\usepackage[a4paper]{geometry}
\usepackage{mathrsfs}
\usepackage[colorlinks]{hyperref}
\usepackage{hyperref}
\renewcommand\eqref[1]{(\ref{#1})}

\makeatletter
\newcommand*{\mint}[1]{%
  \mint@l{#1}{}%
}
\newcommand*{\mint@l}[2]{%
  \@ifnextchar\limits{%
    \mint@l{#1}%
  }{%
    \@ifnextchar\nolimits{%
      \mint@l{#1}%
    }{%
      \@ifnextchar\displaylimits{%
        \mint@l{#1}%
      }{%
        \mint@s{#2}{#1}%
      }%
    }%
  }%
}
\newcommand*{\mint@s}[2]{%
  \@ifnextchar_{%
    \mint@sub{#1}{#2}%
  }{%
    \@ifnextchar^{%
      \mint@sup{#1}{#2}%
    }{%
      \mint@{#1}{#2}{}{}%
    }%
  }%
}
\def\mint@sub#1#2_#3{%
  \@ifnextchar^{%
    \mint@sub@sup{#1}{#2}{#3}%
  }{%
    \mint@{#1}{#2}{#3}{}%
  }%
}
\def\mint@sup#1#2^#3{%
  \@ifnextchar_{%
    \mint@sup@sub{#1}{#2}{#3}%
  }{%
    \mint@{#1}{#2}{}{#3}%
  }%
}
\def\mint@sub@sup#1#2#3^#4{%
  \mint@{#1}{#2}{#3}{#4}%
}
\def\mint@sup@sub#1#2#3_#4{%
  \mint@{#1}{#2}{#4}{#3}%
}
\newcommand*{\mint@}[4]{%
  \mathop{}%
  \mkern-\thinmuskip
  \mathchoice{%
    \mint@@{#1}{#2}{#3}{#4}%
        \displaystyle\textstyle\scriptstyle
  }{%
    \mint@@{#1}{#2}{#3}{#4}%
        \textstyle\scriptstyle\scriptstyle
  }{%
    \mint@@{#1}{#2}{#3}{#4}%
        \scriptstyle\scriptscriptstyle\scriptscriptstyle
  }{%
    \mint@@{#1}{#2}{#3}{#4}%
        \scriptscriptstyle\scriptscriptstyle\scriptscriptstyle
  }%
  \mkern-\thinmuskip
  \int#1%
  \ifx\\#3\\\else_{#3}\fi
  \ifx\\#4\\\else^{#4}\fi
}
\newcommand*{\mint@@}[7]{%
  \begingroup
    \sbox0{$#5\int\m@th$}%
    \sbox2{$#5\int_{}\m@th$}%
    \dimen2=\wd0 %
    \let\mint@limits=#1\relax
    \ifx\mint@limits\relax
      \sbox4{$#5\int_{\kern1sp}^{\kern1sp}\m@th$}%
      \ifdim\wd4>\wd2 %
        \let\mint@limits=\nolimits
      \else
        \let\mint@limits=\limits
      \fi
    \fi
    \ifx\mint@limits\displaylimits
      \ifx#5\displaystyle
        \let\mint@limits=\limits
      \fi
    \fi
    \ifx\mint@limits\limits
      \sbox0{$#7#3\m@th$}%
      \sbox2{$#7#4\m@th$}%
      \ifdim\wd0>\dimen2 %
        \dimen2=\wd0 %
      \fi
      \ifdim\wd2>\dimen2 %
        \dimen2=\wd2 %
      \fi
    \fi
    \rlap{%
      $#5%
        \vcenter{%
          \hbox to\dimen2{%
            \hss
            $#6{#2}\m@th$%
            \hss
          }%
        }%
      $%
    }%
  \endgroup
}
\numberwithin{equation}{section}
\theoremstyle{plain}
\newtheorem{thm}{Theorem}[section]
\newtheorem{prop}[thm]{Proposition}
\newtheorem{cor}[thm]{Corollary}

\theoremstyle{definition}
\newtheorem{defn}[thm]{Definition}
\newtheorem{rem}[thm]{Remark}

\newcommand{\G}{\mathbb{G}}

\title[Anisotropic Shannon inequality]{Anisotropic Shannon inequality}
\author[M. Chatzakou]{Marianna Chatzakou}
\address{
	Marianna Chatzakou:
	\endgraf
    Department of Mathematics: Analysis, Logic and Discrete Mathematics
    \endgraf
    Ghent University, Belgium
  	\endgraf
	{\it E-mail address} {\rm marianna.chatzakou@ugent.be}
		}
	
\author[A. Kassymov]{Aidyn Kassymov}
\address{
  Aidyn Kassymov:
  \endgraf
   \endgraf
  Department of Mathematics: Analysis, Logic and Discrete Mathematics
  \endgraf
  Ghent University, Belgium
  \endgraf
  and
  \endgraf
  Institute of Mathematics and Mathematical Modeling
  \endgraf
  125 Pushkin str.
  \endgraf
  050010 Almaty
  \endgraf
  Kazakhstan
  \endgraf
  and
  \endgraf
  Al-Farabi Kazakh National University
  \endgraf
   71 Al-Farabi avenue
   \endgraf
   050040 Almaty
   \endgraf
   Kazakhstan
  \endgraf
	{\it E-mail address} {\rm aidyn.kassymov@ugent.be} and {\rm kassymov@math.kz}}

\author[M. Ruzhansky]{Michael Ruzhansky}
\address{
  Michael Ruzhansky:
  \endgraf
  Department of Mathematics: Analysis, Logic and Discrete Mathematics
  \endgraf
  Ghent University, Belgium
  \endgraf
 and
  \endgraf
  School of Mathematical Sciences
  \endgraf
  Queen Mary University of London
  \endgraf
  United Kingdom
  \endgraf
  {\it E-mail address} {\rm michael.ruzhansky@ugent.be}
  }

\begin{document}

\thanks{The authors are supported by the FWO Odysseus 1 grant G.0H94.18N: Analysis and Partial Differential Equations and by the Methusalem programme of the Ghent University Special Research Fund (BOF) (Grant number 01M01021). Marianna Chatzakou is a postdoctoral fellow of the Research Foundation – Flanders (FWO) under
the postdoctoral grant No 12B1223N. Michael Ruzhansky and Aidyn Kassymov are also supported by EPSRC grant EP/R003025/2 and the MESRK grant AP19676031, respectively.}

\begin{abstract} 
In this note we prove the anisotropic version of the Shannon inequality. This can be conveniently realised in the setting of Folland and Stein's homogeneous groups. We give two proofs: one giving the best constant, and another one using the Kubo-Ogawa-Suguro inequality. 
\end{abstract}

\maketitle

\tableofcontents

\section{Introduction}

In this paper we derive the logarithmic versions of several well-known functional inequalities. Some inequalities are obtained with best constants, or with semi-explicit constants, the information that is useful for some further applications. Our techniques allow us to derive these inequalities in rather general settings, so we will be working in the settings of general Lie groups, as well as on several classes of nilpotent Lie groups, namely, graded and homogeneous Lie groups. Since on stratified Lie groups we also have the horizontal gradient at our disposal, we will also formulate versions of some of the inequalities in the setting of stratified groups, using the horizontal gradient instead of a power of a sub-Laplacian. 

In the Euclidean space, in one of the Sobolev's pioneering works, Sobolev obtained the following inequality, which at this moment is bearing his name:
\begin{equation}
    \|u\|_{L^{p^{*}}(\mathbb{R}^{n})}\leq C \|\nabla u\|_{L^{p}(\mathbb{R}^{n})},
\end{equation}
where $1<p<n$, $p^{*}=\frac{np}{n-p}$ and $C=C(n,p)>0$ is a positive constant. The best constant of this inequality was obtained by Talenti in \cite{T76}. The Sobolev inequality is one of the most important tools in studying PDE and variational problems.
Folland and Stein extended Sobolev's inequality to general stratified  groups (see e.g. \cite{GV}):
if $\mathbb{G}$ is a stratified group and $\Omega\subset \mathbb{G}$ is an open set, then there exists a constant $C>0$ such that  we have
\begin{equation}\label{Sobolev-Folland-Stein}
\|u\|_{L^{p^{*}}(\Omega)}\leq C\left(\int_{\Omega}|\nabla_{H} u|^{p}dx\right)^{\frac{1}{p}},\; 1<p<Q,\; p^{*}=\frac{Qp}{Q-p},
\end{equation}
for all $u\in C_{0}^{\infty}(\Omega)$. Here $\nabla_{H}$ is the horizontal gradient and $Q$ is the homogeneous dimension of $\mathbb{G}$.
Inequality \eqref{Sobolev-Folland-Stein} is called the Sobolev or Sobolev-Folland-Stein inequality. 
Furthermore, in relation to groups, we can mention
Sobolev inequalities and embeddings on general unimodular Lie groups \cite{VSCC93}, on general locally compact unimodular groups \cite{AR20}, on general noncompact Lie groups \cite{BPTV19,BPV21}, as well as Hardy-Sobolev inequalities on general Lie groups \cite{RY19}.
The Sobolev inequality on graded groups using Rockland operators was proved in \cite{FR17} and the best constant for it was obtained in \cite{RTY20}.

On the other hand, the logarithmic Sobolev inequality was shown to hold on $\mathbb{R}^n$ in the following form: 
\begin{equation}\label{ddlogin}
\int_{\mathbb{R}^{n}}\frac{|u|^{p}}{\|u\|^{p}_{L^{p}(\mathbb{R}^{n})}}\log\left(\frac{|u|^{p}}{\|u\|^{p}_{L^{p}(\mathbb{R}^{n})}}\right)dx\leq\frac{n}{p}\log\left(C\frac{\|\nabla u\|^{p}_{L^{p}(\mathbb{R}^{n})}}{\|u\|^{p}_{L^{p}(\mathbb{R}^{n})}}\right).
\end{equation}
We can refer to \cite{Wei78} for the case $p=2$, but to e.g. \cite{DD03} for some history review of cases 
$1\leq  p<\infty$, including the discussion of best constants.

In \cite{Mer08}, the author obtained a logarithmic Gagliardo-Nirenberg inequality. In \cite {FNQ18} and \cite{KRS20}  the authors proved the logarithmic Sobolev inequality and the fractional logarithmic Sobolev inequality on the Heisenberg group and on homogeneous  groups, respectively. A fractional weighted version of \eqref{ddlogin} on homogeneous groups was proved in \cite{KS20}. In this paper, we prove logarithmic Sobolev inequalities on graded groups and weighted logarithmic Sobolev inequalities on general Lie groups. As applications of these inequalities we show Nash and weighted Nash inequalities on graded and general Lie groups, respectively. The log-Sobolev type inequalities with weights are also sometimes called the log-Hardy inequalities \cite{DDFT10}.

In this paper we establish Shannon's inequality on general homogeneous groups, and we can refer to its links to 
Shannon's entropy \cite{AO73, Isi72, Kap87} and information theory \cite{Sha48, Khi57, MPP00}.

 After Shannon's seminal paper \cite{Sha48} in 1948, several versions of Shannon's inequality have appeared either in discrete, cf. \cite{Khi57,AB00,Wei78,AO73}, or in integral form, cf. \cite{Isi72,Kap87,KOS19}, on certain metric spaces. The underlying motivation is the study of inequalities concerning the entropy function, and, as such, can be regarded as the mathematical foundation of information theory; we refer to \cite{MPP00,MF93} for an overview of the topic. Characterisations of the entropy appear, in the integral form, as the gain of information with functional inequalities. The latter, in the case of a homogeneous Lie group $\G$, with homogeneous dimension $Q$, where $|\cdot|$ an arbitrary homogeneous quasi-norm, and $\alpha\in(1,\infty)$, reads as follows:
For all $u\not=0$ we have 
\begin{equation}\label{EQ:Shannon}
      \int_{\G}\frac{|u(x)|}{\|u\|_{L^{1}(\G)}}\log\left(\frac{|u(x)|}{\|u\|_{L^{1}(\G)}}\right)^{-1}dx\leq \frac{Q}{\alpha}\log\left(\frac{\alpha e A_{Q,\alpha}}{Q}\frac{\||\cdot|^{\alpha}u\|_{L^{1}(\G)}}{\|u\|_{L^{1}(\G)}}\right),
\end{equation}
with an explicit value for $ A_{Q,\alpha}$ (see \eqref{EQ:Aexp1}) that is best possible. Shannon's inequality gives sufficient conditions under which the generalised entropy function, particularly in our case the left-hand side of \eqref{EQ:Shannon}, converges.
Shannon's inequality can be viewed, in some sense, as the counter part of the log-Sobolev inequality as it arises as the limiting case of \eqref{ddlogin} for $p=1$, where, however, instead of the regularity of $u$ it is assumed that $|\cdot|^{\alpha}u$ is in $L^1(\G).$

\section{Preliminaries}
In  this  section,  we  briefly  recall  definitions and  main  properties  of  the homogeneous groups.
The comprehensive analysis on such groups has been initiated in the works of Folland and Stein \cite{FS}, but in our exposition below we follow a more recent presentation in the open access book \cite{FR16}. 

\begin{defn}[\cite{FS, FR16}, Homogeneous group]
A Lie group (on $\mathbb{R}^{N}$) $\mathbb{G}$ with the dilation
$$D_{\lambda}(x):=(\lambda^{\nu_{1}}x_{1},\ldots,\lambda^{\nu_{N}}x_{N}),\; \nu_{1},\ldots, \nu_{n}>0,\; D_{\lambda}:\mathbb{R}^{N}\rightarrow\mathbb{R}^{N},$$
which is an automorphism of the group $\mathbb{G}$ for each $\lambda>0,$
is called a {\em homogeneous (Lie) group}.
\end{defn}
For simplicity,  in this paper we use the notation $\lambda x$ for the dilation $D_{\lambda}(x)$.
We
denote 
\begin{equation}
Q:=\nu_{1}+\ldots+\nu_{N},
\end{equation}
the homogeneous dimension of a homogeneous group $\mathbb{G}$. 
Let $dx$ denote the Haar measure on $\mathbb{G}$ and let $|S|$ denote the corresponding volume of a measurable set $S\subset \mathbb{G}$.
Then we have
\begin{equation}\label{scal}
|D_{\lambda}(S)|=\lambda^{Q}|S| \quad {\rm and}\quad \int_{\mathbb{G}}f(\lambda x)
dx=\lambda^{-Q}\int_{\mathbb{G}}f(x)dx.
\end{equation}
We also note that from \cite[Proposition 1.6.6]{FR16}, the standard Lebesgue measure $dx$ on $\mathbb{R}^{N}$ is the Haar measure on $\G$. 
Then we have the following widely used property in this paper,  see e.g. \cite[p. 19]{RS19}:
Let $\mathbb{G}$ be a homogeneous Lie group with homogeneous dimension $Q$, $r>0$, and let $dx$ be a Haar measure. Then, we have
$$d(rx)=r^{Q}dx.$$
\begin{defn}[{\cite[Definition 3.1.33]{FR16} or \cite[Definition 1.2.1]{RS19}}]\label{quasi-norm}
For any homogeneous group $\mathbb{G}$ there exist  homogeneous quasi-norms, which are  continuous non-negative functions
\begin{equation}
\mathbb{G}\ni x\mapsto |x|\in[0,\infty),
\end{equation}
with the properties

\begin{itemize}
\item[a)] $|x|=|x^{-1}|$ for all $x\in\mathbb{G}$,
\item[b)] $|\lambda x|=\lambda|x|$ for all $x\in \mathbb{G}$ and $\lambda>0$,
\item[c)] $|x|=0$ if and only if $x=0$.
\end{itemize}
\end{defn}
Moreover, the following polarisation formula on homogeneous Lie groups will be used in our proofs, as established by Folland and Stein \cite{FS}.
\begin{prop}[e.g. {\cite[Proposition 3.1.42]{FR16}}]
Let $\mathbb{G}$ be a homogeneous Lie group and $\mathfrak{S}:=\{x\in \mathbb{G}:\,|x|=1\},$ be the unit sphere with respect to the homogeneous quasi-norm $|\cdot|.$ Then there is a unique Radon measure $\sigma$ on $\mathfrak{S}$ such that for all $f\in L^{1}(\mathbb{G}),$ we have
\begin{equation}\label{EQ:polar}
\int_{\mathbb{G}}f(x)dx=\int_{0}^{\infty}
\int_{\mathfrak{S}}f(ry)r^{Q-1}d\sigma(y)dr.
\end{equation}
\end{prop}

\section{Shannon inequality on homogeneous groups}

In this section we show the Shannon inequality on homogeneous Lie groups.  Let us introduce the weighted Lebesgue space $$L^{p,\alpha}(\G):=\{u:u\in L^{p}_{loc}(\G),\,\,\langle x\rangle^{\alpha} u\in L^{p}(\G)\},$$ where $\alpha>0$ and 
$$\langle x \rangle:=(1+|x|^{2})^{\frac{1}{2}},\,\,\,\,\,\text{for}\,\,\,x\in \G,$$
with $|\cdot|$ a homogeneous quasi-norm on $\G.$
Firstly, let us show Shannon inequality.
\begin{thm}[Shannon inequality]\label{shthm1}
Let $\G$ be a homogeneous Lie group with homogeneous dimension $Q$ and let $|\cdot|$ be a homogeneous quasi-norm on $\G$. Suppose that $\alpha\in(0,\infty)$ and $u\in L^{1,\alpha}(\G)\setminus \{0\}$. Then we have 
\begin{equation}\label{sh11}
    \int_{\G}\frac{|u(x)|}{\|u\|_{L^{1}(\G)}}\log\left(\frac{|u(x)|}{\|u\|_{L^{1}(\G)}}\right)^{-1}dx\leq \frac{Q}{\alpha}\log\left(\frac{\alpha e A_{Q,\alpha}}{Q}\frac{\||\cdot|^{\alpha}u\|_{L^{1}(\G)}}{\|u\|_{L^{1}(\G)}}\right),
\end{equation}
where 
\begin{equation}\label{EQ:Aexp1}
    A^{\frac{Q}{\alpha}}_{Q,\alpha}=\frac{|\mathfrak{S}|\Gamma\left({\frac{Q}{\alpha}}\right)}{\alpha},
\end{equation}
with $|\mathfrak{S}|$ the  $Q-1$ dimensional surface measure of the unit quasi-sphere with respect to $|\cdot|$. Moreover, $A_{Q,\alpha}$ is the best possible constant.
This constant is attained with $E_{\alpha}(x)=\exp(-A_{Q,\alpha}|x|^{\alpha})$.
\end{thm}
\begin{proof}
Without loss of generality, it is enough to prove inequality \eqref{sh11} for $\|u\|_{L^{1}(\G)}=1,$ it means, it is enough to prove
\begin{equation}\label{shwith1}
    \int_{\G}|u(x)|\log\left(|u(x)|\right)^{-1}dx\leq \frac{Q}{\alpha}\log\left(\frac{\alpha e A_{Q,\alpha}}{Q}\||\cdot|^{\alpha}u\|_{L^{1}(\G)}\right).
\end{equation}
Let us denote $d\mu=|u(x)|dx$, then we have $\int_{\G}d\mu=1$ is a probability measure. 
First, let us compute the following integral using the change $r^{\alpha}=z$:
\begin{equation}
    \begin{split}
        A^{\frac{Q}{\alpha}}_{Q,\alpha}&=\int_{\G}e^{-|x|^{\alpha}}dx\\&
        \stackrel{\eqref{EQ:polar}}=\int_{0}^{\infty}\int_{\mathfrak{S}}e^{-r^{\alpha}}r^{Q-1}d\sigma(y)dr\\&
        =|\mathfrak{S}|\int_{0}^{\infty}e^{-r^{\alpha}}r^{Q-1}dr\\&
        =\frac{|\mathfrak{S}|}{\alpha}\int_{0}^{\infty}e^{-z}z^{\frac{Q}{\alpha}-1}dz\\&
        =\frac{|\mathfrak{S}|\Gamma\left(\frac{Q}{\alpha}\right)}{\alpha}.
    \end{split}
\end{equation}

By using Jensen's inequality with polarisation and changing variables $A_{Q,\alpha}r^{\alpha}=z$, with $E_{\alpha}(x)=\exp(-A_{Q,\alpha}|x|^{\alpha})$, we compute
\begin{equation}\label{sh321}
    \begin{split}
    \exp\left(\int_{\G}|u(x)|\log\left(\frac{|u(x)|}{E_{\alpha}(x)}\right)^{-1}dx\right)&=\exp\left(\int_{\G}\log\left(\frac{|u(x)|}{E_{\alpha}(x)}\right)^{-1}d\mu\right)\\&
    \leq \int_{\G}\left(\frac{|u(x)|}{E_{\alpha}(x)}\right)^{-1}d\mu\\&
    =\int_{\G}e^{-A_{Q,\alpha}|x|^{\alpha}}dx\\&
    =|\mathfrak{S}|\int_{0}^{\infty}e^{-A_{Q,\alpha}r^{\alpha}}r^{Q-1}dr\\&
    =\frac{|\mathfrak{S}| A^{-\frac{Q}{\alpha}}}{\alpha}\int_{0}^{\infty}e^{-z}z^{\frac{Q}{\alpha}-1}dz\\&
    =\frac{|\mathfrak{S}|\Gamma\left(\frac{Q}{\alpha}\right)A^{-\frac{Q}{\alpha}}}{\alpha}\\&
    =1,
    \end{split}
\end{equation}
then, we obtain
\begin{equation}\label{shin1}
\begin{split}
    \int_{\G}|u(x)|\log\left(|u(x)|\right)^{-1}dx\leq \int_{\G}|u(x)|\log\left(E_{\alpha}(x)\right)^{-1}dx=A_{Q,\alpha}\int_{\G}|x|^{\alpha}|u(x)|dx.
\end{split}
\end{equation}
For $\lambda>0$, let us denote by $u_{\lambda}\in L^{1}(\G)$  the function $u_{\lambda}(x)=\lambda^{Q}u(\lambda x)$. Putting $u_{\lambda}$ in \eqref{shin1} instead of $u$, we have 
\begin{equation}\label{shin12345}
\begin{split}
    \int_{\G}|u_{\lambda}(x)|\log\left(|u_{\lambda}(x)|\right)^{-1}dx\leq A_{Q,\alpha}\int_{\G}|x|^{\alpha}|u_{\lambda}(x)|dx,
\end{split}
\end{equation}
and multiplying both sides by $\frac{\alpha}{Q}$, we have
\begin{equation}\label{shin1234}
\begin{split}
    \frac{\alpha}{Q}\int_{\G}|u_{\lambda}(x)|\log\left(|u_{\lambda}(x)|\right)^{-1}dx\leq \frac{\alpha A_{Q,\alpha}}{Q}\int_{\G}|x|^{\alpha}|u_{\lambda}(x)|dx.
\end{split}
\end{equation}
Then let us compute left hand side of \eqref{shin1234}, 
\begin{equation}\label{shinpr1}
    \begin{split}
        \frac{\alpha}{Q}\int_{\G}|u_{\lambda}(x)|\log(|u_{\lambda}(x)|)^{-1}dx&=\frac{\alpha}{Q}\int_{\G}\lambda^{Q}|u(\lambda x)|\log(\lambda^{Q}|u(\lambda x)|)^{-1}dx\\&
    \stackrel{\eqref{scal}}=\frac{\alpha}{Q}\int_{\G}|u(x)|\log(\lambda^{Q}|u(x)|)^{-1}dx\\&
    =\frac{\alpha}{Q}\int_{\G}|u(x)|\log(|u(x)|)^{-1}dx-\log \lambda^{\alpha},
    \end{split}
\end{equation}
and the right hand side of \eqref{shin1234},
\begin{equation}\label{shinpr2}
    \begin{split}
       \frac{\alpha A_{Q,\alpha}}{Q} \int_{\G}|x|^{\alpha}|u_{\lambda}(x)|dx&
       =\frac{\alpha A_{Q,\alpha}}{Q} \int_{\G}|x|^{\alpha}\lambda^{Q}|u(\lambda x)|dx\\&
       =\frac{\alpha A_{Q,\alpha}}{Q} \int_{\G}\frac{\lambda^{\alpha}}{\lambda^{\alpha}}|x|^{\alpha}\lambda^{Q}|u(\lambda x)|dx\\&
       =\frac{\alpha A_{Q,\alpha}}{Q} \int_{\G}\frac{1}{\lambda^{\alpha}}|\lambda x|^{\alpha}\lambda^{Q}|u(\lambda x)|dx\\&
       \stackrel{\eqref{scal}}=\frac{\alpha A_{Q,\alpha}\lambda^{-\alpha}}{Q}\int_{\G}|x|^{\alpha}|u(x)|dx.
    \end{split}
\end{equation}
Putting the last two facts in \eqref{shin1234}, we get
\begin{equation}
    \frac{\alpha}{Q}\int_{\G}u(x)\log(|u(x)|)^{-1}dx\leq \log\lambda^{\alpha}+\frac{\alpha}{Q}\lambda^{-\alpha}A_{Q,\alpha}\int_{\G}|x|^{\alpha}|u(x)|dx.
    \end{equation}
    Then by taking $\lambda^{\alpha}=\frac{\alpha A_{Q,\alpha}}{Q}\int_{\G}|x|^{\alpha}|u(x)|dx$ in the last fact, we have 
    \begin{equation}
\frac{\alpha}{Q}\int_{\G}|u(x)|\log(|u(x)|)^{-1}dx\leq\log\left(\frac{e\alpha A_{Q,\alpha}}{Q}\int_{\G}|x|^{\alpha}|u(x)|dx\right).
    \end{equation}

Let us prove the best possible constant in \eqref{sh11}.  It is enough to show that the function $E_{\alpha}(x)$ gives equality in \eqref{shin1}, which means that we have
\begin{equation}\label{shpr3}
    \begin{split}
        \frac{\alpha}{Q}\int_{\G}E_{\alpha}(x)\log(E_{\alpha}(x))^{-1}dx&=\frac{\alpha}{Q}\int_{\G}E_{\alpha}(x)\log(\exp(A_{Q,\alpha})|x|^{\alpha})dx\\&
        =\frac{\alpha A_{Q,\alpha}}{Q}\int_{\G}|x|^{\alpha}E_{\alpha}(x)dx. 
    \end{split}
\end{equation}
By taking $E_{\alpha,\lambda}(x)=\lambda^{Q}e^{-A_{Q,\alpha}|\lambda x|^{b}}$ with $\lambda^{b}=\frac{\alpha A_{Q,\alpha}}{Q}\int_{\G}|x|^{\alpha}E_{\alpha}(x)dx$ in \eqref{shpr3}, and repeating same calculation as \eqref{shinpr1} and \eqref{shinpr2}, we get equality in \eqref{shwith1}.
\end{proof}
Let us now show another proof of the Shannon inequality. 
Firstly, we show the Kubo-Ogawa-Suguro inequality and as an application, we derive the Shannon inequality. 

\begin{thm}[Kubo-Ogawa-Suguro inequality] 
Let $\G$ be a homogeneous Lie group with homogeneous dimension $Q$ and a homogeneous quasi-norm $|\cdot|$ on $\G$. Let $\alpha\in(1,\infty)$ and $u\in L^{1,\alpha}(\G)\setminus \{0\}$. Then we have 
\begin{equation}\label{KOSin}
  -  \int_{\G}|u(x)|\log\frac{|u(x)|}{\|u\|_{L^{1}(\G)}}dx\leq Q\int_{\G}|u(x)|\log \left(C_{Q,\alpha}(1+|x|^{\alpha})\right)dx,
\end{equation}
where 
\begin{equation}
    C_{Q,\alpha}=\left(\frac{|\mathfrak{S}|\Gamma\left(\frac{Q}{\alpha}\right)\Gamma\left(\frac{Q}{\alpha'}\right)}{\alpha\Gamma\left(Q\right)}\right)^{\frac{1}{Q}},
\end{equation}
is the best constant with $\frac{1}{\alpha}+\frac{1}{\alpha'}=1$ and $|\mathfrak{S}|$ is the  $Q-1$ dimensional surface measure of the unit quasi-sphere with respect to $|\cdot|$.
\end{thm}
\begin{proof}
Without loss of generality, assume that $\|u\|_{L^{1}(\G)}=1$. Then, by denoting $d\mu=|u(x)|dx$, we have $\int_{\G}d\mu=1$. 
Let us denote by $\varphi(x)=c_{Q,\alpha}(1+|x|^{\alpha})^{-Q}$, where $c_{Q,\alpha}=\frac{\alpha\Gamma\left(Q\right)}{|\mathfrak{S}|\Gamma\left(\frac{Q}{\alpha}\right)\Gamma\left(\frac{Q}{\alpha'}\right)}$ and let us prove that $\|\varphi\|_{L^{1}(\G)}=1$. By using the polar decomposition with the change of variables $(1+r^{\alpha})^{-Q}=t^{Q}$, we compute
\begin{equation}
    \begin{split}
        \int_{\G}(1+|x|^{\alpha})^{-Q}dx&=\int_{0}^{\infty}\int_{\mathfrak{S}}(1+r^{\alpha})^{-Q}r^{Q-1}drd\sigma(y)\\&
        =|\mathfrak{S}|\int_{0}^{\infty}(1+r^{\alpha})^{-Q}r^{Q-1}dr\\&
        =\frac{|\mathfrak{S}|}{\alpha}\int_{0}^{1}(1-t)^{\frac{Q}{\alpha}-1}t^{\frac{Q}{\alpha'}-1}dt\\&
        =\frac{|\mathfrak{S}|}{\alpha}\text{B}\left(\frac{Q}{\alpha},\frac{Q}{\alpha'}\right)\\&
        =\frac{|\mathfrak{S}|\Gamma\left(\frac{Q}{\alpha}\right)\Gamma\left(\frac{Q}{\alpha'}\right)}{\alpha\Gamma\left(Q\right)},
    \end{split}
\end{equation}
where $\text{B}(\cdot,\cdot)$ is the Beta function. Then $\|\varphi\|_{L^{1}(\G)}=1$.
By using this last fact with Jensen's inequality, we get
\begin{equation}
\begin{split}
    \int_{\G}|u(x)|\log\left(\frac{\varphi(x)}{|u(x)|}\right)dx&\leq \log\left(\int_{\G}\frac{\varphi(x)}{|u(x)|}d\mu\right)\\&
    =\log\left(\int_{\G}\varphi(x)dx\right)\\&
    =0.
\end{split}
\end{equation}
It means that  we have
\begin{equation}
\begin{split}
    -\int_{\G}|u(x)|\log|u(x)|dx&\leq -\int_{\G}|u(x)|\log\varphi(x)dx\\&
=-\int_{\G}|u(x)|\log c_{Q,\alpha}(1+|x|^{\alpha})^{-Q}dx\\&
=Q\int_{\G}|u(x)|\log c^{-\frac{1}{Q}}_{Q,\alpha}(1+|x|^{\alpha})dx\\&
=Q\int_{\G}|u(x)|\log C_{Q,\alpha}(1+|x|^{\alpha})dx.
\end{split}
\end{equation}
Also, in the last inequality, equality holds, if and only if 
\begin{equation}
    u(x)=c_{Q,\alpha}(1+|x|^{\alpha})^{-Q}.
\end{equation}
By using Jensen's inequality, we get
\begin{equation}\label{sh1}
    \begin{split}
        \int_{\G}|u(x)|\log(1+|x|^{\alpha})dx&\leq \log\left(\int_{\G}(1+|x|^{\alpha})d\mu\right)\\&
        =\log\left(\int_{\G}|u(x)|(1+|x|^{\alpha})dx\right)\\&
        \leq C\log\left(\int_{\G}\langle x\rangle^{\alpha}|u(x)|dx\right).
    \end{split}
\end{equation}
By using \eqref{sh1}, we have that 
\begin{equation}
\begin{split}
    -\int_{\G}|u(x)|\log|u(x)|dx&\leq Q\int_{\G}|u(x)|\log C_{Q,\alpha}(1+|x|^{\alpha})dx\\&
    =Q\int_{\G}|u(x)|\log (1+|x|^{\alpha})dx+Q\int_{\G}|u(x)|\log C_{Q,\alpha}dx\\&
    \stackrel{\eqref{sh1}}\leq C\log\left(\int_{\G}\langle x\rangle^{\alpha}|u(x)|dx\right)+Q\int_{\G}|u(x)|\log C_{Q,\alpha}dx\\&
    <\infty,
\end{split}
\end{equation}
also implying \eqref{KOSin}.
\end{proof}
Let us show that  Kubo-Ogawa-Suguro inequality also implies Shannon's inequality.
\begin{cor}[Shannon inequality]\label{shthm2}
Let $\G$ be a homogeneous Lie group with homogeneous dimension $Q$ and a homogeneous quasi-norm $|\cdot|$ on $\G$. Let $\alpha\in(1,\infty)$ and $u\in L^{1,\alpha}(\G)\setminus \{0\}$. Then we have 
\begin{equation}\label{shkos}
    \int_{\G}\frac{|u(x)|}{\|u\|_{L^{1}(\G)}}\log\left(\frac{|u(x)|}{\|u\|_{L^{1}(\G)}}\right)^{-1}dx\leq \frac{Q}{\alpha}\log\left(\frac{B_{Q,\alpha}}{\|u\|_{L^{1}(\G)}}\||\cdot|^{\alpha}u\|_{L^{1}(\G)}\right),
\end{equation}
where 
\begin{equation}\label{EQ:Aexp}
    B_{Q,\alpha}=\alpha^{\alpha}(\alpha-1)^{1-\alpha}\left(\frac{|\mathfrak{S}|\Gamma\left(\frac{Q}{\alpha}\right)\Gamma\left(\frac{Q}{\alpha'}\right)}{\alpha\Gamma\left(Q\right)}\right)^{\frac{\alpha}{Q}},
\end{equation}
with $\frac{1}{\alpha}+\frac{1}{\alpha'}=1$ and $|\mathfrak{S}|$ is the  $Q-1$ dimensional surface measure of the unit quasi-sphere with respect to $|\cdot|$.
\end{cor}
\begin{proof}
Similarly to the previous theorem, without loss generality, we can assume that $\|u\|_{L^{1}(\G)}=1$ for  $u\in L^{1}(\G)$. Let us denote $d\mu=|u(x)|dx,$ then we have $\int_{\G}d\mu=1$ is a probability measure. 
By combining \eqref{KOSin} and Jensen's inequality, we get
\begin{equation}\label{sh3}
    \begin{split}
       -  \int_{\G}|u(x)|\log\frac{|u(x)|}{\|u\|_{L^{1}(\G)}}dx&\leq  Q\int_{\G}|u(x)|\log \left(C_{Q,\alpha}(1+|x|^{\alpha})\right)dx\\&
       = Q\int_{\G}\log \left(C_{Q,\alpha}(1+|x|^{\alpha})\right)d\mu\\&
       \leq Q\log\left(\int_{\G} C_{Q,\alpha}(1+|x|^{\alpha})d\mu\right)\\&
       =Q\log\left(\int_{\G} C_{Q,\alpha}|u(x)|(1+|x|^{\alpha})dx\right),
    \end{split}
\end{equation}
where $C_{Q,\alpha}=\left(\frac{|\mathfrak{S}|\Gamma\left(\frac{Q}{\alpha}\right)\Gamma\left(\frac{Q}{\alpha'}\right)}{\alpha\Gamma\left(Q\right)}\right)^{\frac{1}{Q}}.$

For $\lambda>0$, let us denote by $u_{\lambda}\in L^{1}(\G)$ the function $u_{\lambda}(x)=\lambda^{Q}u(\lambda x)$.
Then we have 
\begin{equation*}\label{sh4}
    \begin{split}
        -\int_{\G}|u_{\lambda}(x)|\log|u_{\lambda}(x)|dx&=-\int_{\G}\lambda^{Q}|u(\lambda x)|\log(\lambda^{Q}|u(\lambda x)|)dx\\&
        =-\int_{\G}\lambda^{Q}|u(\lambda x)|\log\lambda^{Q}dx-\int_{\G}\lambda^{Q}|u(\lambda x)|\log|u(\lambda x)|dx\\&
        =   -\log\lambda^{Q}\int_{\G}\lambda^{Q}|u(\lambda x)|dx-\int_{\G}\lambda^{Q}|u(\lambda x)|\log|u(\lambda x)|dx\\&
        =-\log\lambda^{Q}\int_{\G}|u(\lambda x)|d(\lambda x)-\int_{\G}|u(\lambda x)|\log|u(\lambda x)|d(\lambda x)\\&
        =-Q\log\lambda-\int_{\G}|u(x)|\log|u(x)|dx,
    \end{split}
\end{equation*}
and 
\begin{equation*}\label{sh5}
    \begin{split}
        Q\log \left(C_{Q,\alpha}\int_{\G}(1+|x|^{\alpha})|u_{\lambda}(x)|dx\right)&= Q\log \left(C_{Q,\alpha}\int_{\G}\lambda^{Q}(1+|x|^{\alpha})|u(\lambda x)|dx\right)\\&
        =Q\log C_{Q,\alpha}+Q\log \left(\int_{\G}\lambda^{Q}(1+|x|^{\alpha})|u(\lambda x)|dx\right)\\&
        =Q\log C_{Q,\alpha}+Q\log \left(\int_{\G}\lambda^{Q}(1+\frac{\lambda^{\alpha} }{\lambda^{\alpha}}|x|^{\alpha})|u(\lambda x)|dx\right)\\&
    =Q\log C_{Q,\alpha}+Q\log \left(\int_{\G}(1+\lambda^{-\alpha} |\lambda x|^{\alpha})|u(\lambda x)|d(\lambda x)\right)\\&
    =Q\log C_{Q,\alpha}+Q\log \left(\int_{\G}(1+\lambda^{-\alpha} |x|^{\alpha})|u(x)|dx\right)\\&
    =Q\log C_{Q,\alpha}+Q\log \left(1+\lambda^{-\alpha} \| |\cdot|^{\alpha}u\|_{L^{1}(\G)}\right).
    \end{split}
\end{equation*}
Using these two facts in \eqref{sh3}, we get
\begin{equation}
    -\int_{\G}|u(x)|\log|u(x)|dx\leq Q\log C_{Q,\alpha}+Q\log \left(\lambda+\lambda^{1-\alpha} \| |\cdot|^{\alpha}u\|_{L^{1}(\G)}\right).
\end{equation}
By choosing $\lambda=(\alpha-1)^{\frac{1}{\alpha}}\||\cdot|^{\alpha}u\|^{\frac{1}{\alpha}}_{L^{1}(\G)}$, we get 
\begin{equation*}
\begin{split}
     Q\log \left(\lambda+\lambda^{1-\alpha} \| |\cdot|^{\alpha}u\|_{L^{1}(\G)}\right)&=Q\log\left((\alpha-1)^{\frac{1}{\alpha}}\||\cdot|^{\alpha}u\|^{\frac{1}{\alpha}}_{L^{1}(\G)}+(\alpha-1)^{\frac{1-\alpha}{\alpha}}\||\cdot|^{\alpha}u\|^{\frac{1}{\alpha}}_{L^{1}(\G)}\right)\\&
     =Q\log\left(\alpha(\alpha-1)^{\frac{1}{\alpha}-1}\||\cdot|^{\alpha}u\|^{\frac{1}{\alpha}}_{L^{1}(\G)}\right)\\&
     =\frac{Q}{\alpha}\log\left(\alpha^{\alpha}(\alpha-1)^{1-\alpha}\||\cdot|^{\alpha}u\|_{L^{1}(\G)}\right).
\end{split}
\end{equation*}
Finally, we get
\begin{equation}
\begin{split}
    -\int_{\G}|u(x)|\log|u(x)|dx&\leq \frac{Q}{\alpha}\log \left(C^{\alpha}_{Q,\alpha}\alpha^{\alpha}(\alpha-1)^{1-\alpha}\||\cdot|^{\alpha}u\|_{L^{1}(\G)}\right)\\&
    =\frac{Q}{\alpha}\log \left(B_{Q,\alpha}\||\cdot|^{\alpha}u\|_{L^{1}(\G)}\right),
\end{split}
\end{equation}
implying \eqref{shkos}.
\end{proof}

\begin{rem}
 For large $Q\gg 1$, we have that the constant $B_{Q,\alpha}$ in \eqref{shkos} coincides with the best constant $\frac{\alpha A_{Q,\alpha}}{Q}$ in \eqref{sh11}, that is, $$B_{Q,\alpha}\simeq \frac{\alpha e A_{Q,\alpha}}{Q},\,\,\,\,Q\gg 1.$$
\end{rem}
\begin{proof}
From Stirling approximation formula
\begin{equation*}
    \Gamma(Q)\simeq(2\pi)^{\frac{1}{2}}e^{-Q}Q^{Q-\frac{1}{2}},\,\,\,\,\,Q\gg1,
\end{equation*}
we get,
\begin{equation}
    \begin{split}
       \frac{\alpha e A_{Q,\alpha}}{Q B_{Q,\alpha}}&\stackrel{\eqref{EQ:Aexp1},\eqref{EQ:Aexp}}=Q^{-1}\alpha^{1-\alpha}(\alpha-1)^{\alpha-1}e\left(\frac{\frac{|\mathfrak{S}|\Gamma\left({\frac{Q}{\alpha}}\right)}{\alpha}}{\frac{|\mathfrak{S}|\Gamma\left(\frac{Q}{\alpha}\right)\Gamma\left(\frac{Q}{\alpha'}\right)}{\alpha\Gamma\left(Q\right)}}\right)^{\frac{\alpha}{Q}}\\&
       =Q^{-1}(\alpha')^{1-\alpha}e\left(\frac{\Gamma\left(Q\right)}{\Gamma\left(\frac{Q}{\alpha'}\right)}\right)^{\frac{\alpha}{Q}}\\&
       \simeq Q^{-1}(\alpha')^{1-\alpha}e\left(\frac{e^{-Q}Q^{Q-\frac{1}{2}}}{e^{-\frac{Q}{\alpha'}}\left(\frac{Q}{\alpha'}\right)^{\frac{Q}{\alpha'}-\frac{1}{2}}}\right)^{\frac{\alpha}{Q}}\\&
       =(\alpha')^{-\frac{\alpha}{2Q}}\\&
       \stackrel{Q\rightarrow \infty}\rightarrow 1.
    \end{split}
\end{equation}
\end{proof}

\end{document}